\documentclass[12pt]{amsart}
\usepackage{jrmacros}
\usepackage[margin=1.5in]{geometry}
\usepackage[showonlyrefs]{mathtools}

\theoremstyle{plain}
\newtheorem{theorem}{Theorem}[section]
\newtheorem{proposition}[theorem]{Proposition}

\newtheorem{conjecture}[theorem]{Conjecture}

\theoremstyle{definition}
\newtheorem{definition}[theorem]{Definition}

\newtheorem{question}{Question}

\newtheorem*{q1a}{Question 1a}
\newtheorem*{q1b}{Question 1b}

\theoremstyle{remark}
\newtheorem{remark}[theorem]{Remark}

                                        \def\L{\Lambda}
\def\Is{\mathcal{I}}

\def\bs{\mathbf{s}}

\def\qsh{\hat{\mathrm{QSym}}}

\def\bs{\mathbf{s}}

\def\z{\zt}

\def\bs{\mathbf{s}}

\def\gap{\hspace*{5mm}}
\def\ef{\lp\frac{1}{1},\ldots,\frac{1}{p-1}\rp}
\def\wef{\lp\frac{p}{1},\ldots,\frac{p}{p-1}\rp}

\def\ubn{^{[n]}}

\def\tvn{\tilde{\vp}\ubn}
\def\ks{\ker\,\,}
\def\hA{\hat{\cA}}
\def\z{\hat{\vp}}
\def\Ib{I_\beta}          \def\Ibo{\overline{I}_\beta}

\def\H{H_{p-1}}
\def\ack{\section*{Acknowledgements}}

\def\d{$H$-degree}
\def\ks{\mathrm{ker}\,\,}

\def\vp{Z}

\begin{document}
\title[Values of symmetric polynomials]{Values of symmetric polynomials and a truncated analogue of the Riemann zeta function}
\author{Julian Rosen}
\email{j2rosen@uwaterloo.ca}
\address{Department of pure mathematics, University of Waterloo,\newline \indent Waterloo, ON, N2L 3G1 Canada}
\date{\today}
\maketitle
\begin{abstract}
%Harmonic numbers, which are partial sums of the harmonic series, and various generalizations of the harmonic numbers are known to satisfy many congruences. The present work considers congruences of a particular shape, involving values of symmetric functions generalizing the harmonic numbers. The congruences we consider can hold modulo arbitrarily large powers of a prime. We give a conditional classification of all congruences of the shape we consider.
For each positive integer $n$, we determine the set of symmetric functions $f$ for which the congruence ${f(p/1,p/2,\ldots,p/(p-1))\equiv 0 \mod p^n}$ holds for all sufficiently large primes $p$. Our determination is conditional on a conjecture regarding the modulo $p$ independence of Bernoulli numbers.
In a recent work the author introduced a new truncated analogue of the multiple zeta function and investigated a class of relations among values of this function at positive integers. The question answered in the present work is equivalent to the determination of the relations satisfied by values of the corresponding analogue of the ordinary Riemann zeta function.
\end{abstract}

%
%  Section
%
\section{Introduction}
Wolstenholme's congruence, proved in 1862, is the result that the numerator of the harmonic number
\begin{equation*}
\label{eqharm1}
\H:=1+\frac{1}{2}+\frac{1}{3}+\ldots+\frac{1}{p-1}
\end{equation*}
is divisible by $p^2$ for every prime $p\geq 5$. The harmonic number above can be viewed as the value of first elementary symmetric polynomial in $p-1$ variables
\[
e_1=e_1(x_1,\ldots,x_{p-1}):=x_1+\ldots+x_{p-1}\in\Q[x_1,\ldots,x_{p-1}]
\]
evaluated at $x_i=i^{-1}$. More generally the value of the $n$-th elementary symmetric function
%\footnote{The quantity $\displaystyle e_n\ef$ is an example of a \emph{multiple harmonic sum}, and it is sometimes denoted $\H(\{1\}^n)$}
\[
e_n:=\sum_{i_1<\ldots<i_n}x_{i_1}\ldots x_{i_n}
\]
evaluated at $x_i=i^{-1}$ ($i=1,2,\ldots,p-1$) has numerator divisible by $p^2$ or $p$ when $n$ is odd or even respectively (for $p$ is sufficiently large). Similar results are known for the power sum symmetric functions (see \cite{Zha08} for these and many other related congruences). In the present work we investigate congruences for arbitrary symmetric functions. We give an example here.

There are known extensions of the congruences above in terms of Bernoulli numbers:
\begin{equation}
\label{eqz}
e_n\ef\equiv\begin{cases}\frac{-(n+1)}{2(n+2)}p^2\,B_{p-n-2}\mod p^3,\gap n\text{ odd},

\vspace{3mm}

\\
\frac{-1}{n+1}p\,B_{p-n-1}\gap\mod p^2,\gap n\text{ even,}
\end{cases}
p\geq n+3.
\end{equation}
These expressions lead to the congruence, for any integer $n\geq 1$,
\begin{equation}
\label{eqel}
2e_{2n-1}\ef-2n\,e_{2n}\ef\cdot p\equiv 0\mod p^3
\end{equation}
for $p$ sufficiently large, which involves two different elementary symmetric functions. To account for the factor of $p$ in the second term, it is convenient to include an explicit factor of $p^k$ with each term $e_k$, or equivalently to evaluate our symmetric functions at $x_i=pi^{-1}$ (for $i=1,2,\ldots,p-1$). In this way we can express \eqref{eqel} as the congruence\footnote{In fact this congruence holds modulo $p^{2n+3}$.}
\begin{equation}
\label{eqel2}
2\, e_{2n-1}\wef-2n\, e_{2n}\wef\equiv 0\mod p^{2n+2},
\end{equation}
again for $p$ sufficiently large. We consider congruences like \eqref{eqel2}, with the expression $2\,e_{2n-1}-2n\,e_{2n}$ replaced by an arbitrary symmetric function.

\begin{question}
\label{ques1}
For each positive integer $n$, determine the set of symmetric functions $f$ for which the congruence
\begin{equation}
\label{eqf}
f\lp\frac{p}{1},\frac{p}{2},\ldots,\frac{p}{p-1}\rp\equiv 0\mod p^n\tag{$\star$}
\end{equation}
holds for all primes $p$ sufficiently large.
\end{question}

In the present work we give a conditional solution to Question \ref{ques1}.
%We may restrict to the set of $f$ of degree less than $n$, which has finite dimension as a vector space over $\Q$.
We produce for each $n$ an explicit collection of symmetric functions $f$ for which the congruence \eqref{eqf} holds. We show that it would follow from a conjecture on the modulo $p$ independence of Bernoulli numbers that our collection actually contains all $f$ for which \eqref{eqf} holds.

We construct our symmetric functions from certain infinite series identities. An example of such an identity is
\begin{equation}
\label{eqar}
\sum_{k=1}^{\infty}(-1)^k e_k\wef =0,
\end{equation}
which holds for $p\geq 3$ (note that for any fixed $p$ the sum is finite). Reducing modulo $p^n$ shows that the symmetric function
\[
f_n:=\sum_{k=1}^{n-1}(-1)^k e_k
\]
satisfies \eqref{eqf}. We find a family of series identities, of which \eqref{eqar} is the first, and we use these identities to construct our collection of symmetric functions satisfying \eqref{eqf}.
%Conditionally on a conjecture regarding Bernoulli numbers, we show that every symmetric function for which \eqref{eqf} holds arises from our construction.

\subsection{Analogy with the Riemann zeta function}
For $f$ a symmetric function, one can consider the infinite series
\begin{equation}
\label{eqser}
f\lp\frac{1}{1},\frac{1}{2},\frac{1}{3},\ldots\rp.
\end{equation}
Over $\Q$, the algebra of symmetric functions is generated by the power sum symmetric functions $p_n$, $n\geq 1$, and it can be checked that the series \eqref{eqser} converges precisely when $f$ is in the algebra generated by the $p_n$ for $n\geq 2$. In this case the value of \eqref{eqser} is just a polynomial in the values $\zeta(2),\zeta(3),\ldots$ of the Riemann zeta function. We view Question \ref{ques1} as a truncated analogue of asking for a description of the symmetric functions for which \eqref{eqser} vanishes, or equivalently of asking for the algebraic relations over $\Q$ satisfied by the Riemann zeta values.

It is known that
\[
\zeta(2k)=\frac{(-1)^{k+1}B_{2k}(2\pi)^{2k}}{2(2k)!},
\]
where $B_{2k}\in\Q$ are the Bernoulli numbers. This means that $\zeta(2k)$ is a rational multiple of $\zeta(2)^k$. It is conjectured that $\zeta(3),\zeta(5),\ldots$ are algebraically independent over $\Q(\pi^2)$, and this conjecture would imply that \eqref{eqser} vanishes if and only if $f$ is in the ideal generated by the elements \[
p_{2k}-\frac{(-1)^{k+1}24^k\,B_{2k}}{2(2k)!}p_2^k,\gap k\geq 2.
\]

\subsection{Recent related results}
%Several results appearing in the literature answer special cases of Question \ref{ques1}. We mention a couple of these results here.
Many special cases of congruences \eqref{eqf} have appeared in the literature. Tauraso \cite{Tau10} showed that \eqref{eqf} holds for $n=6$ with $f=e_1-e_2+\frac{1}{6}p_3$, where $p_3$ is a power sum symmetric function. Several other similar results are also given in \cite{Tau10}.

Me{\v s}trovi\'c \cite{Mes11} recently
%showed that the congruence
%\[
%{2p-1\choose p-1} \equiv 1-2p\H(1)+4p^2\H(1,1)\mod p^7
%\]
%holds for all primes $p\geq 11$. This extends congruences of Glaisher \cite{}, Zhao \cite{}, and Tauraso \cite{}.
gave a congruence for the binomial coefficient ${2p-1\choose p-1}$ modulo $p^7$ involving multiple harmonic sums. The proof utilizes a number of congruences of the form \eqref{eqf}. Variations on this congruence are given in \cite{Mes12}, Sec.\ 2. Generalizations of Me{\v s}trovi\'c's congruence, holding modulo arbitrarily large powers of $p$, were given by the author in \cite{Ros12a}.

\subsection{Multiple harmonic sums}
A \emph{composition} is a finite ordered list of positive integers. The \emph{weight} of a composition $\bs=(s_1,\ldots,s_k)$ is $w(\bs):=s_1+\ldots+s_k$. For $\bs$ a composition and $n$ a positive integer, the \emph{multiple harmonic sum} is defined by
\[
H_n(\bs):=\sum_{n\geq n_1>\ldots>n_k\geq 1}\frac{1}{n_1^{s_1}\ldots n_k^{s_k}}\in\Q.
\]
%For $k$ a positive integer, we write $\{1\}^k$ for the composition $(1,1,\ldots,1)$ consisting of $k$ 1's. We then have
%\[
%e_k\wef=p^k\H(\{1\}^k).
%\]
We state an equivalent form of Question 1 in terms of multiple harmonic sums:
\begin{q1a}
For each positive integer $n$, determine the set of congruences
\[
\sum_{w(\bs)<n}\alpha_{\bs}\,p^{w(\bs)}\H(\bs)\equiv 0\mod p^n
\]
holding for $p$ sufficiently large, where the coefficients $\alpha_{\bs}$ are rational and satisfy $\alpha_{\bs}=\alpha_{\bs'}$ whenever the composition $\bs'$ is obtained from $\bs$ by rearranging the elements.

\end{q1a}

In a recent work \cite{Ros13}, we consider a more general version of Question 1a with the hypothesis $\alpha_{\bs}=\alpha_{\bs'}$ omitted. This is equivalent to a version of Question 1 with symmetric functions replaced by quasi-symmetric functions.

%
%  Section
%
\section{Algebraic setup}
\label{secalg}

%\subsection{Algebraic setup}
Recall that a \emph{symmetric function} over $\Q$ is a formal power series of bounded degree in countably many variables (with rational coefficients) that is invariant under any permutation of the variables. Given a finite unordered list of rational numbers (or more generally, elements of any $\Q$-algebra), it makes sense to evaluate a symmetric function at these elements. The set of symmetric functions is a commutative ring, which we denote $\L_\Q$ or simply $\L$. The fundamental theorem of symmetric functions says that the elementary symmetric functions freely generate $\L$ as an algebra.

For each prime $p$ there is a ring homomorphism
\begin{align*}
\vp_p:\L&\to\Q,\\
f&\mapsto f\wef.
\end{align*}
%Following a notation used in the study of finite multiple zeta values,
For $n\geq 1$ we set
\[
\cA_n:=\frac{\prod_p \Z/p^n\Z}{\oplus_p \Z/p^n\Z}.
\]
An element of $\cA_n$ is determined by a residue class $a_p\in\Z/p^n\Z$ for all but finitely many $p$, and two families of classes $a_p,a'_p$ determine the same element of $\cA_n$ if and only if $a_p\equiv a'_p\mod p^n$ all but finitely many $p$. For any $f\in\L$ and any prime $p$ not dividing the denominator of any coefficient in $f$, we have $\vp_p(f)\in\Z_{(p)}$. Reducing these elements modulo $p^n$ gives a ring homomorphism
\begin{align*}
\vp\ubn:\L&\to\cA_n\\
f&\mapsto \lp\vp_p(f)\mod p^n\rp.
\end{align*}
The following is an equivalent form of Question \ref{ques1}, stated in terms of $\vp\ubn$:
\begin{q1b}
For each positive integer $n$, describe $\ks\vp\ubn$.
\end{q1b}

%\subsection{Valuation}
The congruence \eqref{eqz} implies that $\vp_p^{[n+1]}(e_{n})=0$ for all $n\geq 1$ (in fact it is true that $\vp_p^{[n+2]}(e_n)=0$ when $n$ is odd, but we will not need this fact). Motivated by this we define a grading on $\L$, which we call the \emph{grading by \d}, by taking $e_n$ to be homogeneous of \d{} $n+1$ and extending multiplicatively (this determines the grading because $\L$ is freely generated as an algebra by the $e_n$). For example, the element $e_1e_2-2e_4$ is homogeneous of \d{} 5.

Take $f\in\L$ and let
\[
f=\sum_{n}f_n
\]
with $f_n$ homogeneous of \d{} $n$. We define
\[
v(f):=\inf \{n:f_n\neq 0\} \in \Z_{\geq 0}\cup\{\infty\}.
\]
We also define for each $n$ an ideal
\[
\Is_n:=\{f\in\L:v(f)\geq n\}.
\]
%\end{definition}

It follows from the congruences \eqref{eqz} that $\Is_n\subset\ker\,\,\vp\ubn$, and we will  denote by $\tvn$ the induced map
\[
\tvn:\L/\Is_n\to\cA_n.
\]
To answer Question \ref{ques1} it suffices to determine $\ks\tvn$. A convenient system of coset representatives of $\Is_n$ in $\L$ is given by the symmetric functions of \d{} less than $n$.
%and our computions will typically be performed in $\L/\Is_n$, which has finite dimension as a vector space over $\Q$. 

\section{A family of congruences}
For $p$ an odd prime, consider the polynomial
\[
f_p(t):={pt-1\choose p-1}=\frac{\big(pt-1\big)\big(pt-2\big)\ldots\big(pt-(p-1)\big)}{(p-1)!}\in\Q[t],
\]
which is seen to satisfy the functional equation $f_p(t)=f_p(1-t)$. We can express $f_p(t)$ in the form
\begin{eqnarray*}
f_p(t)&=&\lp1-\frac{p}{1}\, t\rp\lp1-\frac{p}{2}\,t\rp\ldots\lp 1-\frac{p}{p-1}\,t\rp\\
&=&\sum_{k\geq 0}(-1)^ke_k\wef t^k.
\end{eqnarray*}
Using this expression and computing the coefficient of $t^n$ in the functional equation proves:
\begin{proposition}
\label{propbeta}
For all $k\geq 0$ and all primes $p\geq 3$,
\begin{equation*}
\label{eqprop}
e_k\wef+\sum_{j\geq k}(-1)^{j+1}{j\choose k}e_j\wef=0.
\end{equation*}
\end{proposition}

Note that the sum appearing above is finite, as terms vanish whenever $j\geq p$.
We use Proposition \ref{propbeta} to generate symmetric functions satisfying congruences.

\begin{definition}
\label{defbeta}
For $n,k\geq 0$, define
\[
\beta_k\ubn:=e_k+\sum_{j\geq k} (-1)^{j+1}{j\choose k}e_j\in\L/\Is_n.
\]
This sum is finite, as $e_j\in\Is_n$ once $j$ is sufficiently large. 

For $n\geq 0$, define an ideal
\[
I\ubn:=(\beta_0\ubn,\beta_1\ubn,\ldots)\subset\L/\Is_n.
\]
\end{definition}
We have $\beta_k\ubn=0$ whenever $k\geq n-1$ (or $k\geq n-2$ and $k$ is even). Proposition \ref{propbeta} implies 
\[
\beta_k\ubn\in\ks\tvn
\]
for all $k$ and $n$.
We can now describe our family of congruences.

\begin{theorem}
\label{thunc}
Let $n$ be a non-negative integer and suppose $f\in\L$ satisfies $\overline{f}\in I\ubn$, where $\overline{f}$ is the reduction of $f$ modulo $\Is_n$. Then the congruence
\begin{equation}
\label{eqthm}
f\wef\equiv0\mod p^n
\end{equation}
holds for all $p$ sufficiently large.

\end{theorem}
\begin{proof}
This is equivalent to the inclusion $I\ubn\subset\ks\tvn$, which follows from Proposition \ref{propbeta}.
\end{proof}

%\begin{definition}
%For each non-negative integer $n$, we let $F_n$ denote the set of elements of $\L$ of \d{} less than $n$ whose reductions mod $\Is_n$ are in $I\ubn$.
%\end{definition}
%Theorem \ref{thunc} show that $f\in F_n$ implies the congruence \eqref{eqthm} holds for all sufficiently large $p$. Additionally, for each $n$, $I\ubn$ is finite dimensional as a vector space over $\Q$ and a basis for $I\ubn$ can be computed explicitly. This means a finite computation shows whether a given symmetric functions is in $F_n$.
%
%\begin{example}
%\end{example}
%
%\julian{Maybe give sample computation?}

\section{Structure of Bernoulli numbers modulo $p$}
In this section we discuss a conjecture regarding the structure of the Bernoulli numbers modulo $p$. We begin with an example. Wolsenholme's congruence implies that $e_1\in\ks\vp^{[3]}$. To determine whether $e_1\in\ks\vp^{[4]}$, we need to know whether the congruence
\begin{equation}
\label{eq4}
e_1\wef\equiv0\mod p^4
\end{equation}
holds for all sufficiently large $p$ (we know this holds modulo $p^3$ for $p\geq 5$). Using Eq.\ \eqref{eqz} this is equivalent to asking whether the numerator of the Bernoulli number $B_{p-3}$ is divisible by $p$ for all sufficiently large $p$. It is certainly believed that this should not hold: this would contradict, for example, the conjecture that there are infinitely many regular primes. At present, however, this is not known.

Primes for which $p|B_{p-3}$ are known as Wolstenholme primes and only two are known: 16,843 and 2,124,679. A heuristic argument predicts that there should be infinitely many Wolstenholme primes, and that the number smaller than $x$ should grow like $\log\log x$.

The conjecture that $p\nmid B_{p-3}$ for infinitely many $p$ (which is equivalent \eqref{eq4} failing for infinitely many $p$) has a generalization due to Zhao (\cite{Zha11}, Conjecture 2.1). We state a form of this conjecture here.
\begin{conjecture}
\label{conbnd}
Let $n$ be a positive integer, and suppose ${h\in\Q[x_3,x_5,\ldots,x_{2n+1}]}$ is non-zero and homogeneous (where $\deg(x_{2k+1})=2k+1$). Then there exist infinitely many primes $p$ such that $p$ does not divide the numerator of
\[
h\big( B_{p-3},B_{p-5},\ldots,B_{p-2n-1}\big).
\]
\end{conjecture}

\section{A conditional converse to Theorem \ref{thunc}}
The truth of Conjecture \ref{conbnd} would allow us to make Theorem \ref{thunc} biconditional:
\begin{theorem}
\label{thcon}
Assume the truth of Conjecture \ref{conbnd}. Suppose $f\in\L$ and $n$ is a positive integer, and let $I\ubn$ be given by Definition \ref{defbeta}. Then the congruence
\[
f\wef\equiv0\mod p^n
\]
holds for all $p$ sufficiently large if and only if $\overline{f}\in I\ubn$, where $\overline{f}$ is the reduction of $f$ modulo $\Is_n$.
\end{theorem}

\begin{proof}
This is the statement that $I\ubn=\ks\tvn$. Theorem \ref{thunc} implies $I\ubn\subset\ks\tvn$, so we need to show $\ks\tvn\subset I\ubn$.

Take $f\in\ks\tvn$. We have
\[
\beta_{2k+1}=2e_{2k+1}+\sum_{j\geq 2k+2}(-1)^{j+1}{j\choose 2k+1}e_j,
\]
so we can find an element $g\in\Q[e_2,e_4,\ldots]\subset\L$ of \d{} less than $n$ such that $g\equiv f\mod\Is_n+ I\ubn$. If $g$ is non-zero, take $\tilde{g}$ to be the homogeneous piece of $g$ of lowest \d, and write
\[
\tilde{g}=w(e_2,e_4,e_6,\ldots).
\]
Using \eqref{eqz} we would then get a contradiction to Conjecture \ref{conbnd} by taking
\[
h(x_3,x_5,\ldots)=w\lp\frac{-x_3}{3},\frac{-x_5}{5},\ldots\rp.
\]
We conclude that $g=0$, completing the proof.
\end{proof}

%\section{Example computation}
%In this section we perform a computation for $n=6$. A basis for the space of symmetric functions of \d{} less than 6 is
%\[
%1,\gap e_1,\gap e_2,\gap e_3,\gap e_1^2,\gap e_4,\gap e_1e_2,\gap e_5,\gap e_2^2,\gap e_1e_3,\gap e_1^3
%\]

\section{An analogue of the zeta function}
\label{sec6}
In this section we introduce an analogue of the Riemann zeta function. Many relations among values of this function follow from our congruences for symmetric functions.
\subsection{Completed ring of symmetric functions}
To study congruences modulo all powers of $p$ at once, it is useful to use the completion $\A$ of $\L$ with respect to the grading by \d: it is the projective limit
\[
\A:=\varprojlim_n\L/\Is_n.
\]
It is often convenient to view an element of $\A$ as a formal infinite sum
\begin{equation}
\label{eqdeff}
f=\sum_{n=0}^\infty f_n,
\end{equation}
with $f_n\in\L$ homogeneous of \d{} $n$. We define $\hat{\Is}_k\subset\A$ to be the ideal consisting of those $f$ for which $f_n=0$ for $n<k$. We identify $\L/\Is_n$ with $\A/\hat{\Is}_n$.

\subsection{Zeta function}
The maps $\tvn:\L/\Is_n\to\cA_n$ defined in Sec.\ \ref{secalg} are compatible with the respective quotient maps, so we get a map
\[
\z:\A\to\hA,\gap \text{ where }\hA:=\varprojlim_n \cA_n.
\]
We put the discrete topology on each $\cA_n$ and $\L/\Is_n$ and the projective limit topology on $\hA$ and $\A$, so that $\z$ is continuous. We view $\z$ as an analogue of the Riemann zeta function.

The kernel of $\z$ is a closed ideal of $\A$. The element $f\in\A$ given by \eqref{eqdeff} is in the kernel of $\z$ if and only if for all $N\geq0$ the congruence
\[
\sum_{n<N}f_n\wef\equiv 0\mod p^N
\]
holds for $p$ sufficiently large.

Proposition \ref{propbeta} implies that the elements
\[
\beta_k:=e_k+\sum_{j=k}^{\infty} (-1)^{j+1}{j\choose k}e_j\in\A
\]
are in the kernel of $\z$. If we assume Conjecture \ref{conbnd}, we can show that the $\beta_k$ actually topologically generate $\ks\z$.
\begin{theorem}
\label{thker}
Assume the truth of Conjecture \ref{conbnd}. Then the kernel of $\hat{\vp}$ is equal to the closure of the ideal
\[
\Ib:=(\beta_0,\beta_1,\ldots)\subset\A.
\]
\end{theorem}
\begin{proof}
The inclusion $\Ibo\subset\ks\z$ follows because $\ks\z$ is closed and $\beta_n\in\ks\z$. For the reverse inclusion, suppose $f\in\ks\z$. This means that for each $n$ the reduction of $f$ modulo $\hat{\Is}_n$ is in $\ks\tvn$. Theorem \ref{thcon} then implies that the reduction of $f$ is in $I\ubn$. The preimage of $I\ubn$ under the quotient map $\A\to\A/\hat{\Is}_n$ is $\Ib+\hat{\Is}_n$, so we have $f\in\Ib+\hat{\Is}_n$ for all $n$. The ideals $\hat{\Is}_n$ form a neighborhood basis of 0 in $\A$, so it follows that $f\in\Ibo$.
\end{proof}

\begin{remark}
The ring of symmetric functions has the structure of a Hopf algebra:\ there is a comultiplication map $\Delta:\L\to\L\otimes\L$ giving $\L$ the structure of a cogroup object in the category of $\Q$-algebras. This induces a cogroup object structure on the topological $\Q$-algebra $\A$ (the induced comultiplication $\Delta:\A\to\A\hat{\otimes}\A$ involves the completed tensor product). It can be shown that the closed ideal $\Ibo$ is a Hopf ideal, i.e., the Hopf algebra structure descends to the quotient $\A/\Ibo$. The proof of this fact will be given in a later work.
\end{remark}

We also show that Conjecture \ref{conbnd} implies that every congruence \eqref{eqf} arises from an element of $\ks\z$ in the following sense:

\begin{theorem}
\label{thext}
Assume the truth of Conjecture \ref{conbnd}. Let $n$ be a positive integer, and suppose $f\in\L$ satisfies \eqref{eqf} and that $f$ has \d{} less than $n$. Then there exists
\[
g=\sum_{k\geq 0}g_k\in\ks\z,
\]
with $g_k$ homogeneous of \d{} $k$, such that
\[
f=\sum_{k<n}g_k.
\]
\end{theorem}
This theorem says that every congruence \eqref{eqf} arises from an element of $\ks\z$. Note that the converse of this result is true unconditionally (i.e., the existence of such a $g$ implies that $f$ satisfies \eqref{eqf}). In \cite{Ros13} (Conjecture 1), we conjecture that an analogous result should hold for quasi-symmetric functions.

\begin{proof}
By Theorem \ref{thcon}, the reduction of $f$ modulo $\Is_n$ is in $I\ubn$. This means that we can find elements $r_0,\ldots,r_k\in\L/\Is_n$ such that
\[
f\equiv r_0\beta_0\ubn+\ldots+r_k\beta_k\ubn\mod \Is_n.
\]
Let $\tilde{r}_0,\ldots,\tilde{r}_k\in\L$ be lifts of $r_0,\ldots,r_k$ (we may choose $\tilde{r}_0,\ldots,\tilde{r}_k\in\L$ to be the unique lifts of \d{} less than $n$), and set
\[
g=\tilde{r}_0\beta_0+\ldots+\tilde{r}_k\beta_k\in\A.
\]
Theorem \ref{thker} implies that $g\in\ks\z$, and by construction $g\equiv f\mod\Is_n$. Finally the hypothesis that $f$ has \d{} less than $n$ implies the desired result.
\end{proof}

%\begin{theorem}
%Assume the truth of Conjecture \ref{conbnd}. Then every congruences comes from an asymptotic relation convergent for all $p\geq 3$.
%\end{theorem}
%
%\begin{theorem}
%Assume the truth of Conjecture \ref{conbnd}. Then set of asymptotic relations equals the closure of the ideal generated by the $\beta$.
%\end{theorem}
%
%\begin{theorem}
%Suppose an asymptotic relation has bounded denominators. Then it is correctly convergent for $p\geq 3$.
%\end{theorem}
%Fails for $p=2$.

\subsection{Quasi-symmetric functions and multiple zeta values}
The \emph{ring of quasi-symmetric functions over $\Q$}, denoted $\qs$, is a $\Q$-algebra containing $\L$. In a recent work \cite{Ros13} the author defined a continuous ring homomorphism
\[
\hat{\zeta}:\qsh\to\hA
\]
which is an analogue of the multiple zeta function (here $\qsh$ is the completion of $\qs$). The map $\z$ defined in Sec.\ \ref{sec6} section is the restriction of $\hat{\zeta}$ to $\A\subset\qsh$.% We view $\z$ as a truncated analogue of the Riemann zeta function.

\ack
I thank Jeff Lagarias for many helpful discussion, references, suggestions, and support. This work supported in part by NSF grants  DMS-0943832 and DMS-1101373.

\bibliographystyle{hplain}
\bibliography{jrbiblio}

\end{document}